\documentclass[12pt]{amsproc}
\newtheorem{theorem}{\sc Theorem}[section]
\newtheorem{lemma}[theorem]{\sc Lemma}
\newtheorem{proposition}[theorem]{\sc Proposition}

\newtheorem{Index Convention}{Index Convention}

\begin{document}
\title[On the exponent of a finite group]{On the exponent of a finite group admitting a fixed-point-free four-group of automorphisms}

\author[E. Romano]{Emanuela Romano}
\address[E. Romano]{Dipartimento di Matematica e Informatica\\
  Universit\`a di Salerno\\
  Via Ponte don Melillo\\
  84084 Fisciano (SA) \\  Italy}
\email{eromano@unisa.it}

\author[P. Shumyatsky]{Pavel Shumyatsky}
\address[P. Shumyatsky]{Department of Mathematics\\
  Universidade de Brasilia\\
  70.919 Brasilia - DF\\
  Brazil}
\email{pavel@unb.br}

\thanks{The second author was supported by CNPq-Brazil}

\begin{abstract} Let $A$ be a group isomorphic with either $S_4$, the symmetric group on four symbols, or $D_8$, the dihedral group of order 8. Let $V$ be a normal four-subgroup of $A$ and $\alpha$ an involution in $A\setminus V$. Suppose that $A$ acts on a finite group $G$ in such a manner that $C_G(V)=1$ and $C_G(\alpha)$ has exponent $e$. We show that if $A\cong S_4$ then the exponent of $G$ is $e$-bounded and if $A\cong D_8$ then the exponent of the derived group $G'$ is $e$-bounded. This work was motivated by recent results on the exponent of a finite group admitting an action by a Frobenius group of automorphisms.
\end{abstract}
\maketitle

\section{Introduction}
Let $G$ be a group admitting an action of a group $A$. We
denote by $C_G(A)$ the set $C_G(A)=\{x\in G; x^a=x\mbox{ for any } a\in A\}$,
the centralizer of $A$ in $G$ (the fixed-point group). In many cases the properties of $C_G(A)$ have influence over those of $G$. In particular, it was discovered in the late 90s that the exponent of $C_G(A)$ may have strong impact over the exponent of $G$ \cite{khushu}. Special attention was recently given to the situation where a Frobenius group acts by automorphisms on another group. Recall that a Frobenius group $FH$ with kernel $F$ and complement $H$ can be characterized as a finite group that is a semidirect product of a normal subgroup $F$ by $H$ such that $C_F(h)=1$ for every $h\in H\setminus\{1\}$. By Thompson's theorem \cite{tho} the kernel $F$ is nilpotent, and by Higman's theorem \cite{hi} the nilpotency class of $F$ is bounded in terms of the least prime divisor of $|H|$ (explicit upper bounds for the nilpotency class are due to Kreknin and Kostrikin \cite{kr,kr-ko}). Suppose that the Frobenius group $FH$ acts on a finite group $G$ in such a way that $C_G(F)=1$. It was shown in \cite{khmashu} that in this case the order and rank of $G$ are bounded in terms of $|H|$ and the order and rank of $C_G(H)$, respectively. Further, it was shown that if $F$ is cyclic, then the nilpotency class of $G$ is bounded in terms of $|H|$ and the nilpotency class of $C_G(H)$. In the case when $GF$ is also a Frobenius group with kernel $G$ and complement $F$ (so that $GFH$ is a double Frobenius group) the latter result was obtained earlier in \cite{mash}. This solved in the affirmative Mazurov's problem 17.72(a) in Kourovka Notebook \cite{kour}.

The other problem of Mazurov about double Frobenius groups -- Problem 17.72(b) in Kourovka Notebook -- is whether in a double Frobenius group $GFH$ the exponent of $G$ is bounded in terms of $|H|$ and the exponent of $C_G(H)$ only. That problem seems to be very hard and so far no viable approach to it has been found. We will quote just one result from \cite{khmashu} that indirectly addresses the problem:

\begin{theorem}\label{cikl} Suppose that a Frobenius group
$FH$ with cyclic kernel $F$ and complement $H$ acts on a finite group $G$ in
such a manner that $C_G(F)=1$ and $C_G(H)$ has exponent $e$. Then
the exponent of $G$ is bounded solely in terms of $e$ and $|FH|$.
\end{theorem}

Since the exponent of $G$ here depends on the order of $F$, the above theorem does not yield an answer to Mazurov's problem. The proof of Theorem \ref{cikl} uses Lazard's Lie algebra associated with the Jennings--Zassenhaus filtration and its connection with powerful $p$-groups.

It is natural to ask if the theorem remains valid without assuming that $F$ is cyclic. The case of the smallest Frobenius group whose kernel is non-cyclic was treated in \cite{anovo}. The group in question is of course the non-abelian group of order 12 also known as the Alternating group $A_4$ of degree 4. The following theorem is the main result of \cite{anovo}.

\begin{theorem}\label{main} Suppose that the Frobenius group
$FH$ of order 12 acts coprimely on a finite group $G$ in
such a manner that $C_G(F)=1$ and $C_G(H)$ has exponent $e$. Then
the exponent of $G$ is bounded in terms of $e$ only.
\end{theorem}

Recall that an action of a finite group $A$ on a finite group $G$ is coprime if $(|G|,|A|)=1$.
It is amazing that tools required for the treatment of the situation where $FH$ has order 12 are more sophisticated than those employed in the proof of Theorem \ref{cikl}. In particular, the proof of Theorem \ref{main} uses in the very essential way the solution of the Restricted Burnside Problem \cite{ze1,ze2} while the proof of Theorem \ref{cikl} is based on more simple techniques. Some explanation of this phenomenon can be found in the study of automorphisms of Lie algebras. By the Kreknin Theorem \cite{kr} a Lie algebra that admits a fixed-point-free automorphism of finite order $n$ is soluble with $n$-bounded derived length. On the other hand, a Lie algebra that admits a non-cyclic fixed-point-free group of automorphisms can be unsoluble or soluble with arbitrarily large derived length (see examples in \cite[p. 149--150]{khu}). It is the necessity to work with Lie algebras of unbounded derived length that accounts for the complexity of the proof of Theorem \ref{main}.

The present paper is a natural continuation of \cite{anovo}. Here we study in more detail questions about the exponent of a finite group admitting a fixed-point-free four-group of automorphisms. In the first result that we would like to mention we consider groups acted on by $S_4$, the symmetric group on 4 symbols. In what follows we denote by $V$ the maximal normal 2-subgroup of $S_4$. Of course $V$ is the non-cyclic group of order 4.

\begin{theorem}\label{main1} Let $A$ be isomorphic with $S_4$ and let $\alpha$ be an involution in $A\setminus V$. Suppose that $A$ acts on a finite group $G$ in such a manner that $C_G(V)=1$ and $C_G(\alpha)$ has exponent $e$. Then the exponent of $G$ is bounded in terms of $e$ only.
\end{theorem}

Remark that the above result does not require the coprimeness assumption. It is interesting and somewhat unusual that Theorem \ref{main1} involves only a hypothesis on the exponent of $C_G(\alpha)$ rather than the centralizer of the subgroup of order three as in Theorem \ref{main}. It is well-known that the Sylow 2-subgroup of $S_4$ is isomorphic with $D_8$, the dihedral group   of order 8. When studying the action of $D_8$ on $G$ satisfying the conditions similar to the ones in Theorem \ref{main1} we discovered a new phenomenon -- such an action actually has  strong impact on the exponent of $G'$, the derived group of $G$. Write $D_8$ as a product $V\langle\alpha\rangle$, where $V$ is a four-group and $\alpha$ an involution.

\begin{theorem}\label{main2} Let $A$ be isomorphic with $D_8$ and suppose that $A$ acts on a finite group $G$ in such a manner that $C_G(V)=1$ and $C_G(\alpha)$ has exponent $e$. Then the exponent of $G'$ is bounded in terms of $e$ only.
\end{theorem}

The proofs of both Theorem \ref{main1} and Theorem \ref{main2} use the solution of the Restricted Burnside Problem. Another important tool used in the proof of the above results is the theorem that the exponent of a finite group acted on by a non-cyclic abelian group $A$  is bounded in terms of $|A|$ and the exponents of $C_G(a)$, where $a\in A\setminus\{1\}$ \cite{khushu}. As a part of the proof of Theorem \ref{main2} we obtained the following result that seems to be of independent interest.

\begin{theorem}\label{main3} Let $G$ be a finite group acted on by the four-group $V$ in such a manner that $C_G(V)=1$. Suppose that the centralizers $C_G(v_1)$ and $C_G(v_2)$ of two involutions $v_1,v_2\in V$ have exponent $e$. Then the exponent of $G'$ is bounded in terms of $e$ only.
\end{theorem}

In view of the above results a number of questions about the exponent of a finite group with automorphisms can be asked. In particular, it would be interesting to see if similar results hold in the situation where $V$ is an elementary abelian $p$-group for an odd prime $p$.

Throughout the paper we use the expression ``$(m,n)$-bounded'' for ``bounded above in terms of $m,\,n$ only''.

\section{A criterion of nilpotency for Lie algebras}

In this section we will describe some key Lie-theoretic tools required for the proofs of the main results. In particular we will establish a sufficient condition for a Lie algebra with automorphisms to be nilpotent. Though the hypotheses of Proposition \ref{lie} below may look bizarre, the proposition is sufficient (and perhaps even necessary) for  the group-theoretic applications that will be obtained in Section 5. In what follows the term ``Lie algebra" means a Lie algebra over some commutative ring with unity in which 2 is invertible. If $X\subseteq L$ is a subset of a Lie algebra $L$, we denote by $\langle X\rangle$ the subalgebra generated by $X$. By a commutator of weight 1 in elements of $X$ we mean just any element of $X$. We define inductively commutators in $X$ of weight $w\geq 2$ as elements of the form $[x,y]$, where $x$ and $y$ are commutators in $X$ of weight $w_1$ and $w_2$ respectively such that $w_1+w_2=w$. As usual, $Z(L)$ and $\gamma_i(L)$ denote the center and the $i$th term of the lower central series of $L$, respectively. The centralizer $C_L(S)$ of a subset $S$ is the subalgebra comprised of all elements $x\in L$ such that $[S,x]=0$. If a group $A$ acts by automorphisms on $L$, we denote by $C_L(A)$ the fixed subalgebra of $L$. Recall that an element $a$ of a Lie algebra $L$ is called ad-nilpotent if there exists a positive integer $m$ such that $[x,\underbrace{a,\dots,a}_m]=0$ for all $x\in L$. If $m$ is the least integer with the above property, then we say that $a$ is ad-nilpotent of index $m$.

Let $A=V\langle\alpha\rangle$ be the dihedral group of order 8 with $V=\{1,v_1,v_2,v_3\}$ being a four-group and $\alpha$ an involution such that ${v_1}^\alpha=v_2$. Let $A$ act on a Lie algebra $L$ in such a way that $C_L(V)=0$. For $i=1,2,3$ set $L_i=C_L(v_i)$.
Then we have $L=\bigoplus_{1\leq i\leq 3}L_i$, where $L_i$ are abelian subalgebras with the property that
$$ [L_1,L_2]\leq L_3,\ \ \ [L_2,L_3]\leq L_1,\ \ \ [L_3,L_1]\leq L_2.$$
Of course, the subalgebras $L_i$ are $V$-invariant and $v_j$ acts on $L_i$ by taking every element $x\in L_i$ to $-x$ whenever $i\neq j$. Moreover
$$ {L_1}^\alpha=L_2\ \ \ \text{ and }\ \ \  {L_3}^\alpha=L_3.$$

If $X\subseteq L$ is a subset of $L$, we denote by $I(X)$ the ideal of $L$ generated by $X$
and by $ID(X)$ the minimal $A$-invariant ideal of $L$ containing $X$.

The next two lemmas are taken from \cite{99}.

\begin{lemma}\label{991} (\cite[Lemma 1.2]{99}) Suppose that $[a,b]=0$, where $a\in L_i$, $b\in L_j$ for some $i\neq j$. Then $I([b,L_i])\leq C_L(a)$.
\end{lemma}
\begin{lemma}\label{992} (\cite[Proposition 1.1]{99}) Let $a\in L_1\cup L_2\cup L_3$ and suppose that $a$ is ad-nilpotent of index $m$. Then $I(a)$ is nilpotent of class at most $2m-1$.
\end{lemma}

In the sequel we use the fact that the quotient $L/ID(X)$ naturally satisfies all the necessary assumptions without explicitly mentioning it. More generally, this applies to any quotient over an $A$-invariant ideal. Certainly $\gamma_r(L)$ and $C_L(\gamma_r(L))$ are always $A$-invariant.

We will write $L_\alpha$ for $C_L(\alpha)$. Given a set $X\subseteq L$ such that $L=\langle X\rangle$, an element of $L$ is said to be homogeneous (of weight $w$) with respect to the generating set $X$ if it can be written  as a homogeneous Lie polynomial (of degree $w$) in elements of $X$.

Our main result on nilpotency of Lie algebras is as follows.

\begin{proposition}\label{lie} Let $A=V\langle\alpha\rangle$ be the dihedral group of order 8 acting on a Lie algebra $L$ in such a way that $C_L(V)=0$.  Assume that there exists $x_1\in L_1$ such that $L=\langle x_1,{x_1}^\alpha\rangle$. Moreover, for the generating set $\{x_1,{x_1}^\alpha\}$ there exist positive integers $m$ and $n$ such that every homogeneous element contained in $L_\alpha$ is ad-nilpotent in $L$ of index at most $m$ and every pair of homogeneous elements contained in $L_\alpha$ generates a subalgebra that is nilpotent of class at most $n$. Then $L$ is nilpotent of $(m,n)$-bounded class.
\end{proposition}

First we establish the following related result.

\begin{proposition}\label{4} Assume the hypothesis of Proposition \ref{lie} and let $L$ be soluble with derived length $k$. Then $L$ is nilpotent of $(k,m,n)$-bounded class.
\end{proposition}
\begin{proof} By the main result in \cite{withcarmela2}, $L'$ is nilpotent
of $k$-bounded class. Using the Lie algebra analogue of Hall's
criterion of nilpotency \cite{hall}, we can assume that $L$ is metabelian. Put $x_2={x_1}^\alpha$, $x=x_1+x_2$ and $y=x_1-x_2$. It is clear that $L=\langle x,y\rangle$. Since $x$ is a homogeneous element contained in $L_\alpha$, it follows that $x$ is ad-nilpotent in $L$ of index at most $m$. We also notice that $y=x^{v_1}$. Therefore $y$ is ad-nilpotent in $L$ of index at most $m$, as well. Thus, $L$ is a metabelian Lie algebra generated by two ad-nilpotent elements of index at most $m$. It follows that $L$ is nilpotent of $(k,m,n)$-bounded class. \end{proof}

For every $x\in L$ we write $x=x_1+x_2+x_3$, where $x_i\in L_i$.
We will require the following elementary lemma.

\begin{lemma}\label{rrr} Suppose that $y\in L_i$ for some $i$ and assume that $x\in C_L(y)$. Then also $x_1,x_2,x_3\in C_L(y)$.
\end{lemma}
\begin{proof} We notice that the 1-dimensional subspace $\langle y\rangle$ is $A$-invariant. Therefore the centralizer $C_L(y)$ is $A$-invariant as well. Hence $C_L(y)$ is the direct sum of the subspaces $C_L(y)\cap L_k$ for $k=1,2,3$. The lemma follows.
\end{proof}
\begin{proof}[Proof of Proposition \ref{lie}] Set $x=x_1+x_2$. It is clear that $x\in L_{\alpha}$. By the hypothesis, $x$ is ad-nilpotent in $L$ of index at most $m$. Let $k$ be the minimal number such that $$[x_1-x_2, \underbrace{x,\ldots,x}_k]=0.$$ The proposition will be proved by induction on $k$. We know that $k\leq m$. If $k=1$, then, since 2 is invertible in the ground ring of $L$, it follows that $x_1$ and $x_2$ commute and so $L$ is abelian. Hence, we can assume that $k\geq 2$. In what follows we will call an element $y\in L$ {\it critical} to mean that $L/ID(y)$ is nilpotent of $(m,n)$-bounded class. Thus, the induction hypothesis is that the elements $[x_{1}-x_{2},\underbrace{x,\ldots, x}_{i}]$ are critical for all $i\leq k-1$. Set $t=[x_{1}-x_{2}, \underbrace{x,\ldots,x}_{k-1}]$. If $k-1$ is odd, then $t\in L_3$. Taking into account that $t$ lies in the centralizer of $x$ and using Lemma \ref{rrr}, we conclude that $t$ commutes with $x_{1}$ and $x_{2}$. Hence, $t$ is a critical element that belongs to $Z(L)$ and so the result follows.

Therefore we assume that $k-1$ is even and so $t\in L_1+L_2$. We remark that $t \not\in L_{\alpha}$ because it has the form $l-l^{\alpha}$ for suitable $l \in L$.

Consider the elements $d=t_1+{t_1}^{\alpha}$, $l=t-d$ and $g = [x,d]$. It is easy to see that $d\in L_{\alpha}$, $0\neq l\in L_{2}$ and $g\in L_{3}$.

Suppose first that $g=0$. Then $l$ commutes with $x$. It follows from Lemma \ref{rrr} that $l\in Z(L)$. We remark that $l$ is a critical element. Hence, the result follows.

Therefore we can assume that $g\neq 0$. We will use the number of distinct
commutators in $x$ and $d$ as the second induction parameter. By the hypothesis this number is $n$-bounded. The induction hypothesis will be that every
non-zero commutator in $x$ and $d$ is a critical element. Choose a commutator
$z$ in $x$ and $d$ such that $0\neq z\in Z(\langle x,d\rangle)$. It is clear that $z$ is a critical element. Since both $x$ and $d$ lie in $L_{1}+L_{2}$, every commutator in $x$ and $d$ lies either in $L_3$ or in $L_{1}+L_{2}$. If $z \in L_{3}$, using the fact that $[z,x]=0$ we deduce that $z$ belongs to $Z(L)$ and so the result follows. Suppose that $z=z_{1}+z_{2}\in L_{1}+L_{2}$. Since $z\in L_\alpha$, it follows that $z_{2}=z_{1}^{\alpha}$. Because $g\in L_{3}$ and $[g,z]=0$, Lemma \ref{rrr} implies that $g$ commutes with both $z_{1}$ and $z_{2}$. Moreover, taking into account that $[x_{1},x_{2}]\in L_{3}$ and using Lemma \ref{991} we conclude that $g$ commutes with both ideals $I([x_{1},x_{2},z_{1}])$ and $I([x_{1},x_{2},z_{2}])$.

We remark that $[x_{1},x_{2},z_{1}]-[x_{1},x_{2},z_{2}]\in L_{\alpha}$ and so this element is ad-nilpotent of index at most $m$. Since $$([x_{1},x_{2},z_{1}]-[x_{1},x_{2},z_{2}])^{v_2}=[x_{1},x_{2},z_{1}]+[x_{1},x_{2},z_{2}],$$ we conclude that also $[x_{1},x_{2},z_{1}]+[x_{1},x_{2},z_{2}]$ is ad-nilpotent of index at most $m$.

Set $N=I([x_{1},x_{2},z_{1}])\cap I([x_{1},x_{2},z_{2}])$ and $$J=I([x_{1},x_{2},z_{1}])+ I([x_{1},x_{2},z_{2}])=ID([x_{1},x_{2},z_{1}]).$$
Suppose that $N=0$. In this case $[x_{1},x_{2},z_{1}]$ commutes with $[x_{1},x_{2},z_{2}]$. Since 2 is invertible in the ground ring of $L$ and both $[x_{1},x_{2},z_{1}]+[x_{1},x_{2},z_{2}]$ and $[x_{1},x_{2},z_{1}]-[x_{1},x_{2},z_{2}]$ are ad-nilpotent of index at most $m$, we deduce that $[x_{1},x_{2},z_{1}]$ and $[x_{1},x_{2},z_{2}]$ are ad-nilpotent of index at most $2m-1$.

Lemma \ref{992} shows that both ideals $I([x_1,x_2,z_1])$ and $I([x_1,x_2,z_2])$ are nilpotent of index at most $4m-3$ and so the ideal $J$ is nilpotent of bounded class. Let us pass to the quotient $L/J$ (we use of course that $J$ is $A$-invariant). For simplicity we just assume that $J=0$. Thus both $z_{1}$ and $z_{2}$ commute with $[x_{1},x_{2}]$. By Lemma \ref{991} $z_{1}$ commutes with $I([x_{1},x_{2},x_1])$ and $z_{2}$ commutes with $I([x_{1},x_{2},x_2])$.  Therefore $z$ commutes with the intersection $M=I([x_{1},x_{2},x_{1}])\cap I([x_{1},x_{2},x_{2}])$.

If $M=0$, then $[x_{1},x_{2},x_{1}]$ commutes with $[x_{1},x_{2},x_{1}]^{\alpha} = [x_{2},x_{1},x_{2}]$. With the same argument of above, we deduce that $ID([x_{1},x_{2},x_{1}])$ is nilpotent of bounded class. Passing to the quotient $L/ID([x_{1},x_{2},x_{1}])$ we can assume that $[x_{1},x_{2}]$ commutes with both $x_{1}$ and $x_{2}$. In this case $L$ is nilpotent of class at most 2.  Hence, $L/ID([x_{1},x_{2},x_{1}])$ is nilpotent of class at most 2. In particular we have shown that if $M=0$ then $L$ is soluble with bounded derived length. Proposition \ref{4} yields that if $M=0$ then $L$ is nilpotent of bounded class. Hence, $M$ contains $\gamma_{i}(L)$ for some bounded $i$. Using the fact that $z$ is a critical element that centralizes $M$ we now deduce that the algebra $L$ is soluble with bounded derived length. Proposition \ref{4} now implies that $L$ is nilpotent with bounded class. Recall that we have assumed that $N=0$. Now we will drop this assumption.

Since $N$ is $A$-invariant, we can now conclude that $N$ contains $\gamma_{j}(L)$ for some bounded $j$. Recall that $g$ commutes with both ideals $I([x_{1},x_{2},z_{1}])$ and $I([x_{1},x_{2},z_{2}])$. In particular, it follows that $g$ commutes with $N$. Therefore we found that $g$ is a critical element commuting with $\gamma_{j}(L)$ for some bounded $j$. We conclude that $L$ is soluble with bounded derived length. Another application of Proposition \ref{4} completes the proof.
\end{proof}

\section{Associating Lie algebra to a group}

The proofs of the main results of this paper are based on the so called Lie methods. We will now describe the construction that associates with any group a Lie algebra over the field with $p$ elements. This section does not contain new results and is given here only for the reader's convenience.

Let $G$ be a group and $p$ a prime. We set
$$D_i=D_i(G)=\prod\limits_{jp^k\geq i}\gamma_j(G)^{p^k}.$$ The
subgroups $D_i$ form the Jennings--Zassenhaus filtration
$$G=D_1\geq D_2\geq\cdots $$ of the group $G$. The series satisfies the inclusions $[D_i,D_j]\leq D_{i+j}$ and $D_i^p\leq D_{pi}$ for all $i,j$.
These properties make it possible to construct a Lie algebra
$DL(G)$ over $\mathbb F_p$, the field with $p$ elements. Namely,
consider the quotients $K_i=D_i/D_{i+1}$ as linear spaces over
$\mathbb F_p$, and let $DL(G)$ be the direct sum of these spaces.
Commutation in $G$ induces a binary operation $[\, ,]$ in $DL(G)$.
For elements $xD_{i+1}\in K_i$ and $yD_{j+1}\in K_j$ the operation
is defined by $$[xD_{i+1},yD_{j+1}]=[x,y]D_{i+j+1}\in K_{i+j}$$
and extended to arbitrary elements of $DL(G)$ by linearity. It is
easy to check that the operation is well-defined and that $DL(G)$
with the operations $+$ and $[\, ,]$ is a Lie algebra over
$\mathbb F_p$.

 For any $x\in D_i\setminus D_{i+1}$ let $\bar x$ denote the element $xD_{i+1}$ of $DL(G)$.

\begin{lemma}[Lazard \cite{la}]\label{3.6}
For any $x\in G$ we have $({\rm ad}\, \bar x)^p={\rm ad}\,
\overline{x^p}$. Consequently, if $x$ is of finite order $p^t$,
then $\bar x$ is ad-nilpotent of index at most $p^t$.
\end{lemma}
Denote by $L_p(G)$ the subalgebra of $DL(G)$ generated by $K_1=D_1/D_2$.
The following lemma goes back to Lazard \cite{laz65}; in the present form it can be found, for example, in \cite{khushu}.

\begin{lemma}\label{4.9}
Suppose that $G$ is a $d$-generated finite $p$-group such that the
Lie algebra $L_p(G)$ is nilpotent of class $c$. Then $G$ has a
powerful characteristic subgroup of $(p,c,d)$-bounded index.
\end{lemma}
Recall that powerful $p$-groups were introduced by Lubotzky and
Mann in \cite{lbmn}: a finite $p$-group $G$ is \textit{powerful}
if $G^p\geq [G,G]$ for $p\ne 2$ (or $G^4\geq [G,G]$ for $p=2$).
These groups have many nice properties, so that often a problem
becomes much easier once it is reduced to the case of powerful
$p$-groups. The above lemma is quite useful as it allows us to
perform such a reduction. Precisely, we will require the property
that if $G=\langle g_1,\dots,g_d\rangle$ is a powerful $p$-group
and $e=p^k$, then $G^e=\langle {g_1}^e,\dots,{g_d}^e\rangle$.  Using this it is not really difficult to prove Theorems \ref{main3}, \ref{main2} and \ref{main1} in the particular case
where $G$ is a powerful $p$-group. Thus, we see the general idea of proofs of the main results -- the reduction to the powerful $p$-groups will be performed via Lemma \ref{4.9} while the nilpotency of the corresponding Lie algebras will be established through Proposition \ref{lie}.

\section{Proof of Theorem \ref{main3}}

We start this section with citing some useful results about finite groups
admitting a fixed-point-free four-group of automorphisms. Without
further references we use the well-known fact that if a finite
group $A$ acts coprimely on a finite group $G$ and $N$ is a normal
$A$-invariant subgroup of $G$, then $C_{G/N}(A)=C_G(A)N/N$.

Let $G$ be a finite group and $V=\left\{1,v_1,v_2,v_3\right\}$ the non-cyclic group of order $4$ acting fixed-point-freely on $G$. Then $G$ has odd order. Put $G_i=C_G(v_i)$ for $i\in\{1,2,3\}$. The proofs of the next few lemmas can be found in \cite{Sh2}.

\begin{lemma}\label{112}
Each $G_i$ is abelian and if $i\neq j$ then $v_j$ acts on $G_i$ by
the rule $x^{v_j}=x^{-1}$ for each $x\in G_i$; $i, j \in \left\{1,
2, 3 \right\}$. It follows that every subgroup generated by a
subset of $G_1\cup G_2\cup G_3$ is $V$-invariant.
\end{lemma}

\begin{lemma}\label{113}
Let $x$ be an element of $G$ such that $x^{v_i}=x^{-1}$ for some
$i$. Suppose that $x\in S$, where $S$ is some $V$-invariant
subgroup of $G$. Then there exists a unique pair of elements $y\in
G_j\cap S$ and $t\in G_k\cap S$ such that $x=yty$ and
$\{i,j,k\}=\{1,2,3\}$.
\end{lemma}

\begin{lemma}\label{114} $G=G_1G_2G_3$.\end{lemma}

\begin{lemma}\label{115} $G'=\langle G_1,G_2\rangle\cap\langle G_2,G_3\rangle\cap\langle G_3,G_1\rangle$. In particular, the subgroups $\langle G_i, G_j\rangle$
are normal and contain $G'$.
\end{lemma}

For any $x\in G_i$ and $y\in G_j$ with $i\neq j$ it is
clear that $v_i$ sends $y^x$ to the inverse. Therefore Lemma
\ref{113} guarantees that there exists a unique pair $(s,t)\in G_k
\times G_j$ such that $y^x=sts$ and
$\left\{i,j,k\right\}=\left\{1,2,3\right\}$. Thus we can define
$x*y=s$. According to the same lemma, if $y^x$ is an element of a
$V$-invariant subgroup $S$, then $x*y\in S$. The next lemma is
taken from \cite{withcarmela}.

\begin{lemma}\label{com1}
Let $x\in G_i,y\in G_j$ and $H=\langle x,y\rangle$. Then $\langle
(x*y)^H\rangle=H'$.
\end{lemma}

 Let us define $$R_1=\langle a*b\ |a\in G_2,b\in G_3\rangle,$$
$$R_2=\langle a*b\ |a\in G_1,b\in G_3\rangle,$$
$$R_3=\langle a*b\ |a\in G_1,b\in G_2\rangle$$ and
$$T_1=\langle b*a\ |a\in G_2,b\in G_3\rangle,$$
$$T_2=\langle b*a\ |a\in G_1,b\in G_3\rangle,$$
$$T_3=\langle b*a\ |a\in G_1,b\in G_2\rangle.$$

\begin{lemma}\label{symmetr} If $G$ is nilpotent, we have $R_i=T_i$ for $i=1,2,3$.
\end{lemma}
\begin{proof} Let $G$ be a counterexample of minimal order. If $N$ is any normal $V$-invariant subgroup of $G$, by induction we have $T_iN=R_iN$. Let $Z=Z(G)$. If $Z$ contains a nontrivial element of the form $a*b$ for some $a\in G_i,b\in G_j$, we put $N=\langle a*b\rangle$. It is clear that $N$ is a normal $V$-invariant subgroup contained either in some $T_i$ or in some $R_i$. Since it is central, Lemma \ref{com1} shows that actually $N$ contains both $a*b$ and $b*a$. Hence $N\leq R_i\cap T_i$ and so $R_i=T_i$.

We therefore assume that $Z$ contains no nontrivial elements of the form $a*b$. Let $K=Z_2(G)$ be the second term of the upper central series of $G$ and set $K_i=K\cap G_i$ for $i=1,2,3$. Choose arbitrarily $a\in G_i$ and $b\in K_j$. Then, because of Lemma \ref{com1}, it follows that $a*b\in Z(G)$. Therefore we conclude that $a*b=1$ and so by Lemma \ref{com1} $[a,b]=1$. This happens for every choice of $a\in G_i$ and $b\in K_j$. Hence $K=Z(G)$ and $G$ is abelian, a contradiction.
\end{proof}

We will also require the following result, whose proof is similar to that of Lemma \ref{symmetr}.

\begin{lemma}\label{commu22} If $G$ is nilpotent, then $G_i\cap G'=R_i$ for $i=1,2,3$.\end{lemma}
\begin{proof} Let $G$ be a counterexample of minimal order. Set $K=G'$ and $K_i=K\cap G'$ for $i=1,2,3$. If $N$ is any normal $V$-invariant subgroup of $G$, by induction we have $K_iN=R_iN$. Let $Z=Z(G)$. If $Z$ contains a nontrivial element of the form $a*b$ for some $a\in G_i,b\in G_j$, we put $N=\langle a*b\rangle$. It is clear that $N$ is a normal $V$-invariant subgroup. By Lemma \ref{symmetr} $N$ contained in $R_1\cup R_2\cup R_3$ and in view of the above assumption we get a contradiction. Thus, assume that $Z$ contains no nontrivial elements of the form $a*b$.

Let $T=Z_2(G)$ be the second term of the upper central series of $G$ and set $T_i=T\cap G_i$ for $i=1,2,3$. Choose arbitrarily $a\in G_i$ and $b\in T_j$. Then, because of Lemma \ref{com1}, it follows that $a*b\in Z(G)$. Therefore we conclude that $a*b=1$ and so by Lemma \ref{com1} $[a,b]=1$. This happens for every choice of $a\in G_i$ and $b\in T_j$. Hence $T=Z(G)$ and $G$ is abelian, a contradiction.
\end{proof}

Another useful fact is provided by the next lemma.
\begin{lemma}\label{normg3} Suppose that $N$ is a normal subgroup of $G$ such that $N\leq G_3$.
Then $N\leq Z(G)$.
\end{lemma}
\begin{proof} Choose $b\in N$ and $a\in G_1$ (or $a\in G_2$). Since $b^a\in G_3$, it follows  that $a*b=1$ and so by Lemma \ref{com1} $[a,b]=1$. Thus an arbitrary element of $N$ commutes with an arbitrary element of $G_1\cup G_2$. Using that $G_3$ is abelian and $N\leq G_3$ we conclude that $N\leq Z(G)$.
\end{proof}

Before embarking on the proof of Theorem \ref{main3} we quote the main result of \cite{khushu} that plays a crucial role in subsequent arguments.

\begin{theorem}\label{q2}
Let $q$ be a prime, $e$ a positive integer and $A$ an elementary abelian group of order $q^{2}$. Suppose that $A$ acts as a coprime group of automorphisms on a finite group $G$ and assume that $C_{G}(a)$ has exponent  dividing $e$ for each $a\in A\setminus\{1\}$. Then the exponent of $G$ is $\{e,q\}$-bounded.
\end{theorem}

\begin{proof}[Proof of Theorem \ref{main3}.] Let $K=G'$. Put $G_i=C_G(v_i)$ and $K_i=K\cap G_i$ for $i\in\{1,2,3\}$. By Lemma \ref{114}, we have $G=G_1G_2G_3$ and $K = K_1K_2K_3$. By the hypothesis, the centralizers $G_1$ and $G_2$ have exponent $e$ and therefore the exponents of $K_1$ and $K_2$ divide $e$. In view of Theorem \ref{q2} it is sufficient to prove that $K_3$ has $e$-bounded exponent.

Suppose first that $G$ is a $p$-group. Lemma \ref{commu22} tells
us that $K_3$ is generated by elements of the form $a*b$ for $a\in
G_1$ and $b\in G_2$. Since $K_3$ is abelian, it is sufficient to
show that $a*b$ has $e$-bounded order for every $a\in G_1$ and
$b\in G_2$. Thus, choose $a\in G_1$ and $b\in G_2$ and without
loss of generality we can assume that $G=\langle a,b\rangle$. Let
$L=L_p(G)$. Combining Lemma \ref{3.6} and Lemma \ref{992}, we
deduce that $L$ is nilpotent of $e$-bounded class. Therefore by
Lemma \ref{4.9} $G$ has a powerful characteristic subgroup $H$ of
$e$-bounded index. Since $H$ is $V$-invariant, it is generated by
the centralizers $H\cap G_i$. Hence, we can choose generators
$g_1,\dots,g_d$ of $H$ such that $g_1,\dots,g_d\in G_1\cup G_2\cup
G_3$. Since ${G_1}^e={G_2}^e=1$ and $H^e=\langle
{g_1}^e,\dots,{g_d}^e\rangle$, we conclude that $H^e\leq G_3$.
Thus, Lemma \ref{normg3} tells us that  $H^e\leq Z(G)$ and so
$H/(Z(G)\cap H)$ has $e$-bounded exponent. Therefore, according to
the solution of the restricted Burnside problem, $G/Z(G)$ has
$e$-bounded order and using the Schur Theorem \cite[p. 102]{rob}
we conclude that $G'$ has $e$-bounded order. In particular the
order of $a*b$ is $e$-bounded. Thus, in the case where $G$ is a
$p$-group the theorem is valid. This extends immediately to the
case where $G$ is nilpotent.

In general a finite group admitting a fixed-point-free four-group of automorphisms need not be nilpotent. However it is known that the derived group of such a group is nilpotent \cite[Theorem 10.5.3]{gor}. Therefore we deduce now that $G''$ has $e$-bounded exponent. Passing to the quotient $G/G''$, we can assume that $G$ is metabelian. Because it is sufficient to bound the exponent of $K_3$, we can also assume that $G=\langle G_1,G_2\rangle$. Since $K$ is generated by the centralizers $K_i$ and ${G_1}^e={G_2}^e=1$, it follows that $K^e\leq G_3$. Thus, Lemma \ref{normg3} tells us that  $K^e\leq Z(G)$, and hence $K/(Z(G)\cap K)$ has $e$-bounded exponent. Since $G=\langle G_1,G_2\rangle$, it is  clear that the exponent of $G/K$ divides $e$ and so $G/Z(G)$ has $e$-bounded exponent. If $x,y\in G$, we see that $\langle x,y\rangle/Z(\langle x,y\rangle)$ has $e$-bounded order. Therefore by Schur's theorem $|\langle x,y\rangle'|$ is $e$-bounded, as well. In particular, the order of every commutator $[x,y]$ is $e$-bounded.  Since $G$ is metabelian, it follows that $K$ has $e$-bounded exponent. The proof is now complete.
\end{proof}

\section{Main results}

We will now prove Theorem \ref{main2} and Theorem \ref{main1}.

\begin{proof}[Proof of Theorem \ref{main2}] Recall that the group $A=V\langle\alpha\rangle$,
isomorphic with $D_8$, acts on a finite group $G$. We use the notation introduced in the previous parts of the paper and assume that ${G_1}^\alpha=G_2$ and $G_3$ is $\alpha$-invariant.

Let $H$ be any $A$-invariant subgroup of $G$ and $N$ the minimal $A$-invariant normal subgroup of $H$ containing $C_H(\alpha)$. Then $\alpha$ is fixed-point-free on the quotient $H/N$. It follows that $A$ induces an abelian group of automorphisms of $H/N$ and in particular, since $v_3\in A'$, it follows that $v_3$ must act trivially on $H/N$. The conclusion is that $H=NC_H(v_3)$ and a similar decomposition holds for all $A$-invariant subgroups of $G$ and all $A$-invariant quotients.

In view of Theorem \ref{main3} the exponent of $G'$ is bounded by some number that depends only on the exponents of $G_1$ and $G_2$. Since ${G_1}^\alpha=G_2$, it is sufficient to show that an arbitrary element $x_1$ of $G_1$ has $e$-bounded order. Put ${x_1}^\alpha=x_2$. The subgroup $\langle x_1,x_2\rangle$ is $A$-invariant so without loss of generality we can assume that
$G=\langle x_1,x_2\rangle$. We remark that $v_3$ has no fixed-points in $G/G'$ and recall the formula $H=NC_H(v_3)$. It follows that $G/G'$ is generated by subgroups of the form $(C_{G/G'}(\alpha))^v$ where $v$ ranges through $V$. Hence, $G/G'$ is generated by elements of order dividing $e$. Since the rank of $G/G'$ is at most two, we conclude that the order of $G/G'$ is $e$-bounded.

Suppose first that $G$ is $p$-group and let $L=L_p(G)$. The group
$A$ naturally acts on $L$ and we will show that $L$ satisfies all
the hypothesis of Proposition \ref{lie}. Indeed, the fact that
every homogeneous element contained in $L_\alpha$ is ad-nilpotent
in $L$ of index at most $e$ follows from Lazard's Lemma \ref{3.6},
while the existence of an $e$-bounded number $n$ such that every
pair of homogeneous elements in $L_\alpha$ generates a subalgebra
that is nilpotent of class at most $n$ is immediate from
Zelmanov's solution of the restricted Burnside problem
\cite{ze1,ze2}. Set $L_i=C_{L}(v_{i})$ for each $i\in\{1,2,3\}$.
Since every $L_{i}$ admits a fixed-point-free involutory
automorphism, it follows that every $L_{i}$ is abelian. Thus,
Proposition \ref{lie} tells us that $L$ is nilpotent with
$e$-bounded class. According to Lemma \ref{4.9} $G$ has a powerful
characteristic subgroup $H$ of $e$-bounded index. Using the
formula $H=NC_H(v_3)$, we can choose generators $g_1,\dots,g_d$ of
$H$ each of which belongs either to a conjugate of $C_H(\alpha)$
or to $C_{H}(v_{3})$. Taking into account that $C_H(\alpha)$ has
exponent $e$ and that $H^e=\langle{g_1}^e,\dots,{g_d}^e\rangle$, we derive that $H^e\leq G_3$. It follows now from Lemma \ref{normg3} that $H/(Z(G) \cap H)$ has $e$-bounded exponent, and hence $G/Z(G)$ has $e$-bounded order. Using the
Schur Theorem we conclude that $G'$ has $e$-bounded order. We know
that the order of $G/G'$ is $e$-bounded, as well. So $G=\langle
x_{1},x_{2}\rangle$ has $e$-bounded order, and hence $x_1$ has
$e$-bounded order, too. Since $x_{1}$ was chosen in $G_{1}$
arbitrarily, we conclude that $G_{1}$ has bounded exponent and
hence the same for $G_{2}$.  This proves the theorem in the case
where $G$ is a $p$-group. Of course, from this the case where $G$
is nilpotent is straightforward.

Let us now deal with the case where $G$ is not necessarily nilpotent. We know from \cite[10.5.3]{gor} that $G'$ is nilpotent. Thus, we deduce from the previous paragraph that $G''$ has $e$-bounded exponent. Passing to the quotient $G/G''$, we can assume that $G$ is metabelian. Let
$M=G'$. Since $M=M_1M_2M_3$ where $M_i=M\cap G_{i}$, we see that $M$ is generated by  $M_3$ and some elements of order $e$. Therefore $M^e\leq G_3$. According to Lemma \ref{normg3} it follows that $M^e\leq Z(G)$. As in the proof of Theorem \ref{main3} we conclude that $G'$ has $e$-bounded order. Combining this with the fact that the order of $G/G'$ is likewise
$e$-bounded, the theorem follows. \end{proof}

\begin{proof}[Proof of Theorem \ref{main1}] Since the group $V\langle\alpha\rangle$ is isomorphic with $D_8$, we can use all the information that we obtained in the
proof of Theorem \ref{main2}. It follows that the exponent of $G_1$ is $e$-bounded. Of course in the group $S_4$ all involutions contained in $V$ are conjugate so we conclude that $G_1$, $G_2$
and $G_3$ have the same exponent. Now Theorem \ref{q2} tells us that the exponent of $G$ is $e$-bounded, as required. \end{proof}

\end{document}